\documentclass[preprint,1p]{elsarticle}
 \usepackage{amssymb}
 \usepackage{amsmath}
 \usepackage{amsthm}
\newtheorem{theorem}{Theorem}[section]

\newtheorem{lemma}[theorem]{Lemma}
\newtheorem{remark}[theorem]{Remark}

\begin{document}

\begin{frontmatter}

\title{Global existence and blow-up of solutions for a general class of doubly dispersive nonlocal nonlinear wave equations}

\author{C. Babaoglu$^{1}$}
    \ead{ceni@itu.edu.tr}
\author{H. A. Erbay$^2$\corref{cor1}}
    \ead{husnuata.erbay@ozyegin.edu.tr}
\author{A. Erkip$^3$}
    \ead{albert@sabanciuniv.edu}
\cortext[cor1]{Corresponding author. Tel: +90 216 564 9489 Fax: +90 216 564 9057}

\address{$^1$ Department of Mathematics, Faculty of Science and Letters, Istanbul Technical University,  Maslak 34469, Istanbul,
Turkey}

\address{$^2$ Faculty of Arts and Sciences, Ozyegin University,  Cekmekoy 34794, Istanbul, Turkey}

\address{$^3$ Faculty of Engineering and Natural Sciences, Sabanci University,  Tuzla 34956,  Istanbul,    Turkey}

  \begin{abstract}
  This study deals with the analysis of the Cauchy problem of a general class of  nonlocal nonlinear equations   modeling the
  bi-directional propagation of dispersive waves in various contexts. The nonlocal nature of the problem is reflected by two
  different  elliptic pseudodifferential operators acting on linear and nonlinear functions of the dependent variable,
  respectively. The well-known doubly dispersive nonlinear wave equation that incorporates two types of dispersive effects
  originated from two different dispersion operators falls into the category studied here.  The class of nonlocal nonlinear
  wave equations also covers a variety of well-known wave equations such as  various forms of the Boussinesq equation.
  Local existence of solutions of the Cauchy problem with initial data in suitable Sobolev spaces is proven and the conditions
  for global existence and finite-time blow-up of solutions are established.
 \end{abstract}

 \end{frontmatter}

\setcounter{equation}{0}
\section{Introduction}
\noindent
In this study we mainly establish local existence, global existence and blow-up results for solutions of the Cauchy problem
\begin{eqnarray}
    && u_{tt}-Lu_{xx} =B(g(u))_{xx},\quad x\in \mathbb{R},\quad t>0, \label{non-equ}\\
    && u(x,0) =\varphi (x),\quad u_{t}(x,0)=\psi (x), \label{ini}
\end{eqnarray}
where $g$ is a sufficiently smooth nonlinear function, $L$ and $B$ are linear pseudodifferential operators defined by
\begin{displaymath}
    \mathcal{F}\left( Lv\right)(\xi)=l(\xi) \mathcal{F}(v)(\xi), \quad
    \mathcal{F}\left( Bv\right)(\xi)=b(\xi) \mathcal{F}(v)(\xi).
\end{displaymath}
Here $\mathcal{F}$ denotes the Fourier transform  with respect to variable $x$ and $l(\xi)$ and $b(\xi)$ are the symbols of $L$
and $B$, respectively. We assume that $L$ is an elliptic coercive operator of order $\rho$ with $\rho\ge 0$ while $B$ is an
elliptic positive operator of order $-r$ with $r\ge 0$. In terms of $l(\xi)$ and $b(\xi)$, this means that  there are positive
constants $c_{1}$, $c_{2}$ and $c_{3}$ so that for all $\xi \in  \mathbb{R}$,
\begin{eqnarray}
    && c_{1}^{2}(1+\xi^{2})^{\rho/2} \le l(\xi)\le c_{2}^{2}(1+\xi^{2})^{\rho/2}, \label{con-L}\\
    && 0 < b(\xi)\le c_{3}^{2}(1+\xi^{2})^{-r/2}. \label{con-B}
\end{eqnarray}
We emphasize the fact that, for non-polynomial function $l(\xi)$ or nonzero $b(\xi)$, the equation under investigation is of
nonlocal type. While the operator  $L$ is associated with the regularization resulting from the linear dispersion, the operator
$B$ is associated with the regularization resulting from the smoothing of the non-linear term.  In order to reflect more clearly
the double nature of the dispersive effects, it is convenient to rewrite (\ref{non-equ}) in a slightly different form. Taking
$B=(I+M)^{-1}$ where $I$ is the identity operator  and $M$ is an elliptic positive pseudodifferential operator of order $r>0$,
we rewrite (\ref{non-equ}) in the form
\begin{equation}
    u_{tt}-\tilde{L}u_{xx}+Mu_{tt}=(g(u))_{xx} \label{non-dde}
\end{equation}
with $\tilde{L}=(I+M)L$.  The second and third terms on the left-hand side of this equation represent two sources of dispersive
effects. The relation $\xi \mapsto \omega^{2} (\xi)= \xi^{2}\tilde{l}(\xi)/\left(1+m(\xi)\right)$ where $\tilde{l}(\xi)$ and
$m(\xi)$ are the symbols of $\tilde{L}$ and
$M$, respectively, will be referred to as the linear dispersion relation for (\ref{non-dde}). Since the symbols of $\tilde{L}$
and $M$ will appear in the numerator and denominator, respectively, of the linear dispersion relation,  we informally describe
the two dispersive effects as  "numerator-based" dispersive effect and a "denominator-based" dispersive effect to emphasize the
double nature of dispersion.

Even though our main interest lies primarily in understanding the role of pseudodifferential operators $L$, $B$, it is worth
noting that when $l(\xi)$ is a polynomial, $L$ becomes a differential operator and similarly that, when $b(\xi)$ equals the
reciprocal of a polynomial, $B$ becomes the Green function of the corresponding differential operator. In the polynomial case,
the equation under investigation (that is, (\ref{non-equ}) or (\ref{non-dde})) turns out to be some well-known nonlinear wave
equations for suitable choices of the operators  $\tilde{L}$ and $M$. For instance, we may note that, with the substitution
$\tilde{L}=1-\partial_{x}^{2}$ and $M=-\partial_{x}^{2}$, (\ref{non-dde}) reduces  to the so-called double dispersion equation
\begin{equation}
    u_{tt}-u_{xx}-u_{xxtt}+u_{xxxx}=(g(u))_{xx}.  \label{dde}
\end{equation}
 This equation is the most familiar example or special case of (\ref{non-equ}) and was derived in many different contexts (see,
 for instance, \cite{samsonov1,samsonov2} where it describes the propagation of longitudinal strain waves in a nonlinearly elastic
 rod). Thus, (\ref{non-dde}) might be referred to as a natural generalization of the double dispersion equation through the
 nonlocal operators $L$ and $B$.

 We also point out that  (\ref{non-dde}) reduces to the Boussinesq equation
\begin{equation}
    u_{tt}-u_{xx}+u_{xxxx} =(g(u))_{xx}  \label{boussinesq}
\end{equation}
with the substitution $\tilde{L}=1-\partial_{x}^{2}$ and $M=0$ (the zero operator), while it becomes the improved (or regularized)
Boussinesq equation
\begin{equation}
    u_{tt}-u_{xx}-u_{xxtt}=(g(u))_{xx}  \label{imp-boussinesq}
\end{equation}
with the substitution $\tilde{L}=I$ and $M=-\partial_{x}^{2}$ \cite{boussinesq1,boussinesq2}. Also, assuming $L=0$ and considering the operator $B$ as a convolution
\begin{displaymath}
    (Bv)(x)=(\beta\ast v)(x)=\int \beta(x-y)v(y) dy
\end{displaymath}
with the kernel $\beta(x)=\mathcal{F}^{-1}\left(b(\xi)\right)$ where $\mathcal{F}^{-1}$ denotes the inverse Fourier transform, we observe that (\ref{non-equ}) reduces to
\begin{equation}
    u_{tt}=\left( \int \beta(x-y)g(u(y,t))dy\right)_{xx}.  \label{nne}
\end{equation}
This equation was derived in \cite{duruk1} to model the propagation of strain waves in a one-dimensional, homogeneous, nonlinearly
and nonlocally elastic infinite medium (see \cite{duruk2} and \cite{erbay} for its coupled form and two-dimensional form,
respectively). Our inspiration for the present study comes essentially from (\ref{nne}) modelling an integral-type nonlocality of
elastic materials. In the present study we add to (\ref{nne})  the other type of nonlocality, originating from the inclusion of linear higher order gradients,  and focus on how the qualitative results obtained for (\ref{nne}) in \cite{duruk1} carry over to (\ref{non-equ}).

There is quite extensive literature on the well-posedness of the Cauchy problem for the Boussinesq equation (\ref{boussinesq})
(see e.g., \cite{bona,tsutsumi,liu1,xue}), for the improved Boussinesq equation (\ref{imp-boussinesq}) and its higher order
generalizations (see e.g., \cite{turitsyn,liu2,chen-wang,constantin,wang-mu,duruk3}) and for the double dispersion equation
(\ref{dde}) (see e.g., \cite{wang-chen}). In \cite{duruk1} consideration was given to the well-posedness of the Cauchy problem
for the nonlocal equation (\ref{nne}). The question that naturally arises is under which conditions the Cauchy problem
(\ref{non-equ})-(\ref{ini}) is well-posed and this is the subject of the present study.

The paper is organized as follows: To simplify the presentation, through Sections 2 to 5,  the special case where $B$ is the
identity operator will be treated and the modifications that would be needed for the general case will be given in Section 6.
That is, in Sections 2-5 the Cauchy problem for the equation
\begin{equation}
    u_{tt}-Lu_{xx} =(g(u))_{xx} \label{non-boussinesq}
\end{equation}
is considered only; while the Cauchy problem (\ref{non-equ})-(\ref{ini}) is considered in Section 6.   In Section 2, the
required a priori estimates are established for the linearized version of the Cauchy problem. In Section 3, the local existence
and uniqueness for the nonlinear Cauchy problem is proven using the contraction mapping principle. The main theorems stating the
global existence and uniqueness of the solution are demonstrated in Section 4. The blow-up criteria is presented in Section 5.
Finally, in Section 6, the global existence and blow-up results obtained through Sections 2 to 5 are extended to the Cauchy
problem (\ref{non-equ})-(\ref{ini}).

Throughout the study, we use the following notational conventions: $H^s=H^s(\mathbb{R})$ will denote the $L^2$ Sobolev space
on $\mathbb{R}$. For the $H^s$ norm, we use the Fourier transform representation
$\left\Vert u \right\Vert _s^2=\int(1+\xi^2)^s |\widehat u(\xi)|^2 d\xi$ where the symbol $\widehat ~$ represents the Fourier
transform. We use $ \left\Vert u \right\Vert_\infty$, $\left\Vert  u \right\Vert$ and  $\langle u,v \rangle$ to denote the
$L^\infty$ and $L^2$ norms and the inner product in $L^2$, respectively.

\setcounter{equation}{0}
\section{Cauchy Problem for the Linearized Equation}
\noindent
In this section we focus on the linear version of  (\ref{non-boussinesq}) and prove the following existence and uniqueness result.
\begin{theorem}\label{theo2.1}
    Let $~T>0$, $~s\in \mathbb{R}$, $~\varphi \in H^{s}$, $~\psi \in H^{s-1-{\rho\over 2}}$ and
    $~h\in L^{1}( \left[0,T\right],H^{s+1-{\rho\over 2}})$. Then the Cauchy problem
    \begin{eqnarray}
    && u_{tt}-Lu_{xx} =(h(x,t))_{xx},\quad x\in \mathbb{R},\quad t>0, \label{lin-equ}\\
    && u(x,0) =\varphi (x),\quad u_{t}(x,0)=\psi (x), \label{lin-ini}
    \end{eqnarray}
    has a unique solution $~u\in C( \left[0,T\right],H^{s})\cap C^{1}( \left[0,T\right],H^{s-1-{\rho\over 2}})$ satisfying the following
    estimate
    \begin{equation}
     \left\Vert u(t) \right\Vert_{s}+\left\Vert u_{t}(t)\right\Vert_{s-1-{\rho\over 2}}
         \leq  (A_{1}+A_{2}T) \left(
            \left\Vert \varphi \right\Vert_{s}+\left\Vert \psi\right\Vert_{s-1-{\rho\over 2}}
            + \int_{0}^{t}\left\Vert h(\tau) \right\Vert_{s+1-{\rho\over 2}}d\tau \right)  \label{lin-est}
    \end{equation}
    for $0\le t \le T$, with some positive constants $A_{1}$ and $A_{2}$.
\end{theorem}
\begin{proof}
    To make calculations easier to follow, we introduce the notation $K=L^{1/2}$ and  $k(\xi) = \sqrt{l(\xi)}$. Then
    the symbol $k(\xi)$ of the operator $K$ satisfies
    \begin{equation}
        c_{1}(1+\xi^{2})^{\rho/4} \le k(\xi)\le c_{2}(1+\xi^{2})^{\rho/4}. \label{con-K}
    \end{equation}
    Applying the Fourier transform to (\ref{lin-equ})-(\ref{lin-ini}) gives the initial-value problem
    \begin{eqnarray*}
    && \widehat{u}_{tt}+\left(\xi k(\xi)\right)^{2}\widehat{u} =-\xi^{2}\widehat{h}(\xi,t),\\
    && \widehat{u}(\xi,0) =\widehat{\varphi}(\xi),\quad \widehat{u}_{t}(\xi,0)=\widehat{\psi}(\xi)
    \end{eqnarray*}
    for the corresponding non-homogeneous ordinary differential equation. The solution of the initial-value problem is
    \begin{equation}
        \widehat{u}(\xi,t)=\widehat{\varphi}(\xi)\cos(\xi k(\xi)t)+{\widehat{\psi(\xi)}\over {\xi k(\xi)}}\sin(\xi k(\xi)t)
            -\int_{0}^{t}{\xi \over {k(\xi)}}\sin(\xi k(\xi)(t-\tau))\widehat{h}(\xi,\tau)d\tau.  \label{ode-sol}
    \end{equation}
    We see that this solution is generated by the semigroup
    \begin{displaymath}
        {\cal S}(t)v={\cal F}^{-1}\left( {{\sin(\xi k(\xi)t)}\over {\xi k(\xi)}}\widehat{v}(\xi) \right),
    \end{displaymath}
    which allows us to rewrite (\ref{ode-sol}) in the form  (note that ${\cal S}(t)$ commutes with differentiation)
    \begin{equation}
        u(t)=\partial_{t}{\cal S}(t)\varphi+{\cal S}(t)\psi+\int_{0}^{t}\partial_{x}^{2}{\cal S}(t-\tau)h(\tau)d\tau.
        \label{ode-sol-r}
    \end{equation}
    Now, we will estimate the terms on the right-hand side of (\ref{ode-sol-r}) separately. The estimate for the first term is
    \begin{equation}
     \left\Vert \partial_{t}{\cal S}(t)v \right\Vert_{s}^{2}
     = \int_{\mathbb{R}}(1+\xi^2)^s \cos^{2}(\xi k(\xi)t)|\widehat v(\xi)|^2 d\xi
       \le \int_{\mathbb{R}}(1+\xi^2)^s |\widehat v(\xi)|^2 d\xi =\left\Vert v \right\Vert_{s}^{2}. \label{first-term}
    \end{equation}
    For the second term, we have
    \begin{eqnarray*}
     && \!\!\!\!\!\!\!\!\!\!\!\!\!
     \left\Vert {\cal S}(t)v \right\Vert_{s}^{2}
     = \int_{\mathbb{R}}(1+\xi^2)^s \left( {{\sin(\xi k(\xi)t)}\over {\xi k(\xi)}}\right)^{2}|\widehat v(\xi)|^2 d\xi
        \nonumber \\
     && \!\!\!\! = \int_{|\xi|<1}(1+\xi^2)^s
        \left( {{\sin(\xi k(\xi)t)}\over {\xi k(\xi)}}\right)^{2}|\widehat v(\xi)|^2 d\xi
        +\int_{|\xi|\ge 1}(1+\xi^2)^s
        \left( {{\sin(\xi k(\xi)t)}\over {\xi k(\xi)}}\right)^{2}|\widehat v(\xi)|^2 d\xi.
    \end{eqnarray*}
    Note that  $~\sin^{2}(\xi k(\xi)t)\le (\xi k(\xi)t)^{2}$ for $~|\xi|<1$, while $~\sin^{2}(\xi k(\xi)t)\le 1$ and
    ${1\over \xi^{2}}\le {2\over {1+\xi^{2}}}$  for $|\xi|\ge 1$. Using these inequalities and the assumption
    (\ref{con-K}) we get
     \begin{eqnarray}
     &&  \!\!\!\!\!\!\!\!\!\!\!\!\!
      \left\Vert {\cal S}(t)v \right\Vert_{s}^{2}
          \leq  t^{2}\int_{|\xi|<1}(1+\xi^2)^{s-1-{\rho\over 2}}(1+\xi^2)^{{\rho\over 2}+1}|\widehat v(\xi)|^2 d\xi
         +{2\over c_{1}^{2}}\int_{|\xi|\ge 1}(1+\xi^2)^{s-1-{\rho\over 2}}|\widehat v(\xi)|^2 d\xi
         \nonumber \\
     &&~~~~~ \leq  \left(t^{2}2^{1+{\rho\over 2}}+{2\over c_{1}^{2}}\right)\int_{\mathbb{R}}(1+\xi^2)^{s-1-{\rho\over 2}}|\widehat v(\xi)|^2 d\xi
    \end{eqnarray}
    which leads to
    \begin{equation}
       \left\Vert {\cal S}(t)v \right\Vert_{s}^{2}\le \left(t^{2}2^{1+{\rho\over 2}}+{2\over c_{1}^{2}}\right)
       \left\Vert v \right\Vert_{s-1-{\rho\over 2}}^{2}.
        \label{est-b}
    \end{equation}
    The third term can be estimated via Minkowski's inequality for integrals and (\ref{est-b}) to get
    \begin{eqnarray}
      \left\Vert \int_{0}^{t}\partial_{x}^{2}{\cal S}(t-\tau)v(\tau)d\tau \right\Vert_{s}
         &\le & \int_{0}^{t} \left\Vert \partial_{x}^{2}{\cal S}(t-\tau)v(\tau) \right\Vert_{s}d\tau \nonumber \\
        &\le & \int_{0}^{t} \left\Vert {\cal S}(t-\tau)v(\tau) \right\Vert_{s+2}d\tau \nonumber \\
        &\le & \int_{0}^{t}  \left( (t-\tau)^{2}2^{1+{\rho\over 2}}+{2\over c_{1}^{2}}\right)^{1/2}\left\Vert v(\tau) \right\Vert_{s+1-{\rho\over 2}}d\tau \nonumber \\
        &\le &   \left(t^{2}2^{1+{\rho\over 2}}+{2\over c_{1}^{2}}\right)^{1/2}\int_{0}^{t} \left\Vert v(\tau) \right\Vert_{s+1-{\rho\over 2}}d\tau .
        \label{int-term}
    \end{eqnarray}
    Summing up the estimates (\ref{first-term}), (\ref{est-b}) and (\ref{int-term}) in (\ref{ode-sol}), we obtain
    \begin{equation}
     \left\Vert u(t) \right\Vert_{s}\leq  \left\Vert \varphi \right\Vert_{s}
            + \left(t^{2}2^{1+{\rho\over 2}}+{2\over c_{1}^{2}}\right)^{1/2} \left( \left\Vert \psi\right\Vert_{s-1-{\rho\over 2}}
            +  \int_{0}^{t}\left\Vert h(\tau) \right\Vert_{s+1-{\rho\over 2}}d\tau \right)
            \label{u-est}
    \end{equation}
    for $0\le t \le T$. On the other hand, differentiating (\ref{ode-sol}) with respect to $t$, we get
    \begin{equation}
        u_{t}(t)=\partial_{t}^{2}{\cal S}(t)\varphi+\partial_{t}{\cal S}(t)\psi
            +\int_{0}^{t}\partial_{x}^{2}\partial_{t}{\cal S}(t-\tau)h(\tau)d\tau.
        \label{time-der}
    \end{equation}
    We have the estimate
    \begin{eqnarray*}
      \left\Vert \partial_{t}^{2}{\cal S}(t)v \right\Vert_{s-1-{\rho\over 2}}^{2}
                &=& \int_{\mathbb{R}}(1+\xi^2)^{s-1-{\rho\over 2}}(\xi k(\xi))^{2} \sin^{2}(\xi k(\xi)t)|\widehat v(\xi)|^2 d\xi  \\
                & \le & c_{2}^{2}\int_{\mathbb{R}}(1+\xi^2)^s {\xi^{2}\over {1+\xi^{2}}} |\widehat v(\xi)|^2 d\xi  \\
                & \le & c_{2}^{2}\int_{\mathbb{R}}(1+\xi^2)^s |\widehat v(\xi)|^2 d\xi
                =  c_{2}^{2} \left\Vert v \right\Vert_{s}^{2}
    \end{eqnarray*}
    for the first term of (\ref{time-der}). For the second and third terms of it we use the estimates
    $\left\Vert \partial_{t}{\cal S}(t)v \right\Vert_{s} \le \left\Vert v \right\Vert_{s}$ and
    \begin{eqnarray*}
        \left\Vert \int_{0}^{t}\partial_{x}^{2}\partial_{t}{\cal S}(t-\tau)v(\tau)d\tau \right\Vert_{s-1-{\rho\over 2}}
         &\le &
         \left\Vert \int_{0}^{t}\partial_{x}^{2}\partial_{t}{\cal S}(t-\tau)v(\tau)d\tau \right\Vert_{s}  \\
         &\le &   \left(t^{2}2^{1+{\rho\over 2}}+{2\over c_{1}^{2}}\right)^{1/2}\int_{0}^{t} \left\Vert v(\tau) \right\Vert_{s+1-{\rho\over 2}}d\tau,
    \end{eqnarray*}
    respectively. The use of these estimates in (\ref{time-der}) leads to
    \begin{equation}
     \left\Vert u_{t}(t) \right\Vert_{s-1-{\rho\over 2}}\leq c_{2} \left\Vert \varphi \right\Vert_{s}
            +  \left\Vert \psi\right\Vert_{s-1-{\rho\over 2}}
            + \int_{0}^{t}\left\Vert h(\tau) \right\Vert_{s+1-{\rho\over 2}}d\tau
            \label{ut-est}
    \end{equation}
    for $0\le t \le T$.
    Adding up  (\ref{u-est}) and (\ref{ut-est}) then gives us the required estimate (\ref{lin-est}).
\end{proof}

\setcounter{equation}{0}
\section{Local Existence for the Nonlinear Problem}
\noindent
In this section, we prove  local (in time) existence and uniqueness of solutions for the Cauchy problem given by
(\ref{non-boussinesq}) and (\ref{ini}). To do this, we use the contraction mapping principle.  First, we assume that
the initial data is such that $~\varphi \in H^{s}$ and $~\psi \in H^{s-1-{\rho\over 2}}$ for some fixed $s>1/2$. Then,
for a fixed $T>0$, we define the Banach space
\begin{equation}
     X(T)=\{u\in C( \left[0,T\right],H^{s})\cap C^{1}( \left[0,T\right],H^{s-1-{\rho\over 2}})\}
            \label{banach-space}
\end{equation}
endowed with the norm
\begin{equation}
     \left\Vert u \right\Vert_{X(T)}
        =\max_{t\in \left[0,T\right]} \left( \left\Vert u(t) \right\Vert_{s}+\left\Vert u_{t}(t)\right\Vert_{s-1-{\rho\over 2}}\right).
        \label{banach-norm}
\end{equation}
At this point, we need to remind the reader that, by the Sobolev Embedding Theorem,  $H^{s}(\mathbb{R})\subset L^{\infty}(\mathbb{R})$ for $s>1/2$. This fact allows us to deduce that $u\in C( \left[0,T\right],L^{\infty})$ whenever
$u\in X(T)$. Also, the fact that, by the Sobolev Embedding Theorem, there is a constant $d$ such that $\left\Vert u(t) \right\Vert_{L^{\infty}} \le d \left\Vert u(t) \right\Vert_{X(T)}$ (for $s>1/2$) will be used in what follows.

We now define a closed subset $Y(T)$ of $X(T)$ as follows
\begin{equation}
     Y(T)=\{u\in X(T) : \left\Vert u \right\Vert_{X(T)}\le A\}
            \label{banach-space}
\end{equation}
for some constant $A > 0$ to be determined later. Consider the initial-value problem
\begin{eqnarray}
    && u_{tt}-Lu_{xx} =(g(w))_{xx},   \label{banach-equ}\\
    && u(x,0) =\varphi (x),\quad u_{t}(x,0)=\psi (x) \label{banach-ini}
\end{eqnarray}
with $w\in Y(T)$. Then, Theorem \ref{theo2.1} implies that the problem (\ref{banach-equ})-\ref{banach-ini}) has a unique solution $u(x,t)$. The map that carries $w\in Y(T)$ into the unique solution $u(x,t)$ of (\ref{banach-equ})-\ref{banach-ini})
will be denoted by ${\cal S}$; that is, $u(x,t)={\cal S}(w)$. We now prove  that, for appropriately chosen $T$ and $A$, the map ${\cal S}$ has a unique fixed point in $Y(T)$. This will be done in three steps. In the first step we establish that
the range of $Y(T)$ under the map ${\cal S}$ belongs to the space $X(T)$. Secondly, we derive suitable estimates on $\left\Vert {\cal S}(w) \right\Vert_{X(T)}$ so that $ {\cal S}(Y(T))\subset Y(T)$. The third step is to show that the mapping ${\cal S}$ is a contraction mapping.

Now we state the following two lemmas, which will be used to bound the nonlinear term in what follows.
\begin{lemma}\label{lem3.1}\cite{runst}
     Assume that $f\in C^{k}(\mathbb{R})$, $f(0) = 0$, $u \in H^{s}\cap L^{\infty}$ and $k = [s]+1$,
    where $s\ge 0$. Then, we have
    \begin{displaymath}
     \left\Vert f(u) \right\Vert_{s} \le C_{1}(M)  \left\Vert u \right\Vert_{s}
    \end{displaymath}
    if $\left\Vert u \right\Vert_{L^{\infty}}\le M$, where $C_{1}(M)$ is a constant dependent on $M$.
\end{lemma}
\begin{lemma}\label{lem3.2}\cite{runst}
    Assume that $f\in C^{k}(\mathbb{R})$,  $u, v \in H^{s}\cap L^{\infty}$ and $k = [s]+1$,
    where $s\ge 0$. Then, we have
    \begin{displaymath}
     \left\Vert f(u)-f(v) \right\Vert_{s} \le C_{2}(M)  \left\Vert u-v \right\Vert_{s}
    \end{displaymath}
    if $\left\Vert u \right\Vert_{L^{\infty}}\le M$, $\left\Vert v \right\Vert_{L^{\infty}}\le M$,
    $\left\Vert u \right\Vert_{s}\le M$, and $\left\Vert v \right\Vert_{s}\le M$, where $C_{2}(M)$ is a constant dependent on $M$
    and $s$.
\end{lemma}
The following  lemma, which provides a simple sufficient condition for the map  to be a contraction mapping, is the key to the local  existence and uniqueness theorem for the Cauchy problem given by (\ref{non-boussinesq}) and (\ref{ini}).
\begin{lemma}\label{lem3.3}
    Assume that $\rho \ge 2$, $s > 1/2$, $~\varphi \in H^{s}$, $~\psi \in H^{s-1-{\rho\over 2}}$ and
    $~g\in C^{[s]+1}(\mathbb{R})$. Then for suitably chosen $A$ and sufficiently small $T$, the map $\cal S$ is
    a contractive mapping from $Y(T)$ into itself.
\end{lemma}
\begin{proof}
To prove the lemma, we need to show that  $ {\cal S}(Y(T))\subset Y(T)$. For that, we use a standard contraction argument.
Let $w\in Y (T)$ be given. The Sobolev Embedding Theorem implies that there is some constant $d$ so that
$~\left\Vert w(t) \right\Vert_{L^{\infty}} \le d \left\Vert w(t) \right\Vert_{s}~$ for $t\in \left[0, T\right]$.
Since $~\left\Vert w(t) \right\Vert_{s}\le \left\Vert w \right\Vert_{X(T)}~$ and
$~\left\Vert w \right\Vert_{X(T)}\le A$, this inequality becomes $~\left\Vert w(t) \right\Vert_{L^{\infty}} \le d A$.
Then by Lemma \ref{lem3.1}
    \begin{equation}
         \left\Vert g(w(t)) \right\Vert_{s} \le C_{1}(dA) \left\Vert w(t) \right\Vert_{s} \label{non-inequality}
    \end{equation}
where $C_{1}(dA)$ is a constant dependent on both $d$ and $A$. Taking $h(x,t) \equiv g(w(x,t))$ in Theorem \ref{theo2.1} we
observe that the solution $u ={\cal S}(w)$ of the problem (\ref{banach-equ})-(\ref{banach-ini}) belongs to
$~C(\left[0,T\right],H^{s})\cap C^{1}( \left[0,T\right],H^{s-1-{\rho\over 2}})~$ and that
\begin{displaymath}
     \left\Vert u(t) \right\Vert_{s}+\left\Vert u_{t}(t)\right\Vert_{s-1-{\rho\over 2}}
         \leq  (A_{1}+A_{2}T) \left(
            \left\Vert \varphi \right\Vert_{s}+\left\Vert \psi\right\Vert_{s-1-{\rho\over 2}}
            + \int_{0}^{t}\left\Vert g(w(\tau)) \right\Vert_{s+1-{\rho\over 2}}d\tau \right).
    \end{displaymath}
Then this inequality and  (\ref{banach-norm}) yield
\begin{equation}
\left\Vert {\cal S}(w) \right\Vert_{X(T)} \le (A_{1}+A_{2}T) \left(\left\Vert \varphi \right\Vert_{s}+\left\Vert \psi\right\Vert_{s-1-{\rho\over 2}}
                        +T \max_{t\in \left[0,T\right]}\left\Vert g(w(\tau))\right\Vert_{s+1-{\rho\over 2}}\right). \label{s-inequality}
 \end{equation}
Noting that $s\pm 1-{\rho\over 2}\le s$ for $\rho\ge 2$, the inequality (\ref{non-inequality}) and (\ref{banach-norm}) lead to
\begin{displaymath}
       \max_{t\in \left[0,T\right]}\left\Vert g(w(\tau))\right\Vert_{s+1-{\rho\over 2}}
            \le C_{1}(dA) \left\Vert w(t) \right\Vert_{X(T)}.
\end{displaymath}
Using this inequality in (\ref{s-inequality}) and keeping in mind that $~\left\Vert w \right\Vert_{X(T)}\le A$, we obtain
\begin{displaymath}
      \left\Vert {\cal S}(w) \right\Vert_{X(T)}
                 \leq    (A_{1}+A_{2}T) \left(\left\Vert \varphi \right\Vert_{s}+\left\Vert \psi\right\Vert_{s-1-{\rho\over 2}}
                        +T C_{1}(dA) A\right).
    \end{displaymath}
Consequently, the inequality  $\left\Vert {\cal S}(w) \right\Vert_{X(T)}\le A$ holds, provided that
\begin{equation}
    (A_{1}+A_{2}T) \left(a+T C_{1}(dA) A\right)\le A, \label{sss}
\end{equation}
where, to shorten expressions, we denote by $a$ the norm of the initial data, namely
 $\left(\left\Vert \varphi \right\Vert_{s}+\left\Vert \psi\right\Vert_{s-1-{\rho\over 2}}\right)=a$.
Since the choice of our constant $A$ depends on a choice of $a$, we note that the condition under which $\cal S$  is a
contraction mapping depends on the choice of the norm of the initial data, i.e.,
$\left(\left\Vert \varphi \right\Vert_{s}+\left\Vert \psi\right\Vert_{s-1-{\rho\over 2}}\right)$.

Let $A=\lambda a$ where $\lambda$ is a constant to be chosen later. Expanding the parentheses in (\ref{sss}) and
substituting  $A=\lambda a$ into the
resulting inequality we get
\begin{displaymath}
    A_{1}a+aT [A_{2}+A_{1}\lambda C_{1}(d\lambda a)+T A_{2}\lambda C_{1}(d\lambda a)]\le \lambda a.
\end{displaymath}
If we set $\lambda=A_{1}+1$, this inequality becomes
\begin{displaymath}
T [A_{2}+A_{1}\lambda C_{1}(d\lambda a)+T A_{2}\lambda C_{1}(d\lambda a)]\le 1.
\end{displaymath}
Note that this inequality holds if $T$ is small enough. For a suitable choice of $A$ and $T$ we have
$\left\Vert {\cal S}(w) \right\Vert_{X(T)}\le A$ and hence $ {\cal S}(Y(T))\subset Y(T)$.

Now, let $w, \tilde{w}\in Y (T)$  with $ u={\cal S}(w)$, $ \tilde{u}={\cal S}(\tilde{w})$. Set  $V=u-\tilde{u}$, $W=w-\tilde{w}$.
Then $V$ satisfies (\ref{banach-equ}) with zero initial data:
\begin{eqnarray}
    && V_{tt}-LV_{xx} =(g(w)-g(\tilde{w}))_{xx},   \label{banach-equ-a}\\
    && V(x,0) =0,\quad V_{t}(x,0)=0. \label{banach-ini-a}
\end{eqnarray}
Using the estimate (\ref{lin-est}) of Theorem \ref{theo2.1} gives
\begin{equation}
      \left\Vert V(t) \right\Vert_{s}+\left\Vert V_{t}(t)\right\Vert_{s-1-{\rho\over 2}}
                \le    (A_{1}+A_{2}T) \int_{0}^{t} \left\Vert g(w(\tau))-g(\tilde{w}(\tau)) \right\Vert_{s+1-{\rho\over 2}} d\tau .
                \label{v-estimate}
\end{equation}
 Note that $s+1-{\rho\over 2}\le s$  for $\rho\ge 2$ and consequently
 $ \left\Vert g(w(\tau))-g(\tilde{w}(\tau)) \right\Vert_{s+1-{\rho\over 2}}
        \le \left\Vert g(w(\tau))-g(\tilde{w}(\tau)) \right\Vert_{s} $.
            On the other hand, by Lemma \ref{lem3.2}, we have
\begin{displaymath}
  \left\Vert g(w(\tau))-g(\tilde{w}(\tau)) \right\Vert_{s+1-{\rho\over 2}}
        \le C_{2}(dA)\left\Vert w(\tau)-\tilde{w}(\tau) \right\Vert_{s} = C_{2}(dA) \left\Vert W(\tau)\right\Vert_{s},
\end{displaymath}
where $C_{2}(dA)$ is a constant dependent on both $d$ and $A$. Substitution of this inequality into (\ref{v-estimate}) yields
 \begin{equation}
      \left\Vert V(t) \right\Vert_{s}+\left\Vert V_{t}(t)\right\Vert_{s-1-{\rho\over 2}}
                \le    (A_{1}+A_{2}T) C_{2}((A_{1}+1)da)T \max_{t\in \left[0,T\right]}\left\Vert W(t) \right\Vert_{s},
                \label{}
\end{equation}
where we have used  $A=\lambda a=(A_{1}+1)a$. Using this inequality and the norm defined in (\ref{banach-norm}) one gets
\begin{displaymath}
    \left\Vert V \right\Vert_{X(T)} \le
        (A_{1}+A_{2}T) C_{2}\left((A_{1}+1)da\right)T \left\Vert W \right\Vert_{X(T)}.
\end{displaymath}
We now choose $T$ small enough so that $(A_{1}+A_{2}T) C_{2}((A_{1}+1)da)T\le {1\over 2}$. With that choice of $T$ the mapping
$\cal S$ becomes contractive. This completes the proof of the lemma.
\end{proof}

The Banach Fixed Point Theorem states that every contraction mapping has a unique fixed point. Since the map $\cal S$ is a
contraction from a closed subset $Y(T)$ of the Banach space $X(T)$ into $Y(T)$ by Lemma \ref{lem3.3}, there is a
unique $u\in Y(T)$ such that $S(u) = u$. We have thus proved the following local existence and uniqueness result
for the nonlinear Cauchy problem.
\begin{theorem}\label{theo3.4}
    Assume that $\rho \ge 2$, $s>{1\over 2}$, $~\varphi \in H^{s}$, $~\psi \in H^{s-1-{\rho\over 2}}$ and
    $~g\in C^{[s]+1}(\mathbb{R})$. Then there is some $T>0$
     such that the Cauchy problem given by (\ref{non-boussinesq}) and (\ref{ini}) has a unique solution in
     $~C( \left[0,T\right],H^{s})\cap C^{1}( \left[0,T\right],H^{s-1-{\rho\over 2}})$.
\end{theorem}
\begin{remark}\label{rem3.5}
The condition $s>1/2$ is necessary for the control of the nonlinear term $g(u)$ through Lemmas \ref{lem3.1} and \ref{lem3.2}.
On the other hand, note that the Sobolev space $~ H^{s-1-{\rho\over 2}}$ where $\psi$ and $u_{t}(t)$ lie may have a negative exponent.
\end{remark}

\setcounter{equation}{0}
\section{Global Existence for the Nonlinear Problem}
\noindent
In this section, we will prove that,  under suitable assumptions on the initial data, a unique solution to the nonlinear Cauchy problem given by
(\ref{non-boussinesq}) and (\ref{ini}) exists for all $t\in [0,\infty)$. The basic idea behind that proof is to show that the local solution of Section
3 can be extended uniquely to $[0,\infty)$.   The main ingredient is the following  theorem.

\begin{theorem}\label{theo4.1}
     Assume that $\rho \ge 2$, $s>{1\over 2}$, $~\varphi \in H^{s}$, $~\psi \in H^{s-1-{\rho\over 2}}$ and $~g\in C^{[s]+1}(\mathbb{R})$ and that the unique
     solution of the  Cauchy problem is defined on the maximal time interval $[0, T_{\max})$. If the maximal time is finite, i.e. $T_{\max}<\infty$, then
     \begin{displaymath}
    \limsup_{t \rightarrow T_{\max}^{-}}~[ \left\Vert u(t) \right\Vert_{s}+\left\Vert u_{t}(t)\right\Vert_{s-1-{\rho\over 2}}]=\infty .
    \end{displaymath}
\end{theorem}
\begin{proof}
    The main approach, which we use below, is to look for the local solutions of the Cauchy problems on finite time intervals and then is to  patch  those
    local solutions together in a continuous manner.  Suppose that $u$ is the unique local solution of the Cauchy problem given by
    (\ref{non-boussinesq}) and (\ref{ini}) for $[0, T_{1}]$. We then consider the following Cauchy problem
    \begin{eqnarray*}
        && u_{tt}-Lu_{xx} =(g(u))_{xx}, \quad x\in \mathbb{R},\quad t>T_{1},  \\
        && u(x,T_{1}) =\varphi_{1} (x),\quad u_{t}(x,T_{1})=\psi_{1} (x),
    \end{eqnarray*}
    where $\varphi_{1}\in H^{s}$ and  $\psi_{1}\in H^{s-1-{\rho\over 2}}$. Note that the initial data of this shifted problem
    are obtained from the unique solution on $[0, T_{1}]$, that is, $\varphi_{1}=u(x,T_{1})$ and $\psi_{1}=u_{t}(x,T_{1})$.
    Applying  Theorem \ref{theo3.4} to the shifted problem, we see that there exists a unique solution on an interval $[T_{1}, T_{2}]$ for some
    $T_{2} > T_{1}$. Patching the two local solutions yields an extended unique local solution on $[0, T_{2}]$. Repeating this process $i$ times iteratively,
    we can extend the unique solution  to the interval $[0, T_{i+1}]$ as long as   the conditions  $u(x,T_{j})= \varphi_{j} \in H^{s}$ and
    $u_{t}(x,T_{j})=\psi_{j}\in H^{s-1-{\rho\over 2}}$ for $j=1,2,..i$ hold. This implies that, as long as
    $\limsup_{t \rightarrow T^{-}} [ \left\Vert u(t) \right\Vert_{s}+\left\Vert u_{t}(t)\right\Vert_{s-1-{\rho\over 2}}] <\infty$,
    we can extend the solution to $[0, T)$ for any $T$, using a similar approach. We then conclude that the solution cannot be extended beyond some
    finite $T_{\max}$ if and only if
    $\limsup_{t \rightarrow T_{\max}^{-}} [ \left\Vert u(t) \right\Vert_{s}+\left\Vert u_{t}(t)\right\Vert_{s-1-{\rho\over 2}}]=\infty$.
\end{proof}
We now prove that the unique solution of the Cauchy problem satisfies an energy identity.
\begin{theorem}\label{theo4.2}
    Suppose that $u(x,t)$ is a solution of the Cauchy problem given by (\ref{non-boussinesq}) and (\ref{ini}) on some interval $[0, T)$.  Let $K=L^{1/2}$, $G(u) =\int_{0}^{u}g(p)dp$ and $\Lambda^{-\alpha}w={\cal F}^{-1}[ |\xi|^{-\alpha}{\cal F}w]$, where ${\cal F}$ and ${\cal F}^{-1}$ denote Fourier transform and inverse Fourier transform   in the $x$-variable, respectively. If $\Lambda^{-1}\psi \in L^{2}$, $K\varphi \in L^{2}$ and $G(\varphi) \in L^{1}$, then, for any $t\in [0, T)$, the energy identity
    \begin{equation}
     E(t)=\left\Vert \Lambda^{-1}u_{t}(t) \right\Vert^{2}+\left\Vert Ku(t) \right\Vert^{2}+2\int_{\mathbb{R}}G(u(x,t))dx=E(0),
        \label{energy}
    \end{equation}
    is satisfied.
\end{theorem}
\begin{proof}
    Applying the Fourier transform to  (\ref{non-boussinesq}) and using the definition of $\Lambda$ in the resulting equation yields
    \begin{equation}
    \Lambda^{-2}u_{tt}+K^{2}u+g(u)=0. \label{equ-motion}
    \end{equation}
    If we multiply both sides of this equation by $u_{t}$ and then integrate over $\mathbb{R}$ with respect to $x$, we get
    \begin{equation}
     \langle \Lambda^{-2}u_{tt}+K^{2}u+g(u), u_{t} \rangle =0.
        \label{inner}
    \end{equation}
    Since $\Lambda^{-1}$  and $K$ are self-adjoint operators,  (\ref{inner}) becomes
    \begin{displaymath}
     \langle \Lambda^{-1}u_{tt}, \Lambda^{-1}u_{t} \rangle + \langle Ku, Ku_{t} \rangle + \langle g(u), u_{t} \rangle=0
    \end{displaymath}
    from which we deduce
    \begin{displaymath}
     {1\over 2}{d\over {dt}}\left(
        \left\Vert \Lambda^{-1}u_{t} \right\Vert^{2}+\left\Vert Ku \right\Vert^{2}
        +2\int_{\mathbb{R}}\left( \int_{0}^{u}g(p)dp\right)dx \right)=0.
    \end{displaymath}
    And from this equation we get  (\ref{energy}). To justify rigorously the above formal computation, first we note that
    $\Lambda^{-1}\psi \in L^{2}$, $K\varphi \in L^{2}$  imply $\psi \in H^{-1}$, $\varphi \in H^{\rho/2}$, respectively.
    By Theorem \ref{theo3.4} we have $u(t)\in H^{\rho/2}$ and thus $Ku(t)\in L^{2}$, $G(u(t))\in L^{1}$ for all $t\in [0,t)$.
    An argument similar to that in Lemma 4.1 of \cite{duruk1} shows that $\Lambda^{-1}u(t)\in L^{2}$.
\end{proof}
The energy  identity leads to global existence through the following theorem when $s={\rho\over 2}$.
\begin{theorem}\label{theo4.3}
    Assume that $\rho \ge 2$,  $g\in C^{[{\rho \over 2}]+1}(\mathbb{R})$, $\varphi \in H^{{\rho \over 2}}$,
    $\psi\in H^{-1}$, $\Lambda^{-1}\psi \in L^{2}$, $G(\varphi) \in L^{1}$ and $G(u) \ge 0$ for all $u\in \mathbb{R}$. Then
    the  Cauchy problem given by (\ref{non-boussinesq}) and (\ref{ini}) has a unique global solution
    $u\in C\left( [0,\infty),H^{{\rho \over 2}}\right)\cap C^{1}\left( [0,\infty),H^{-1}\right)$.
\end{theorem}
\begin{proof}
    Suppose the solution exists on some interval $[0, T)$. If $G(u)\ge 0$, it follows from (\ref{energy}) that
    \begin{displaymath}
        \left\Vert \Lambda^{-1}u_{t} \right\Vert^{2}+\left\Vert Ku \right\Vert^{2} \le E(0).
     \end{displaymath}
     Furthermore, we have
     \begin{eqnarray}
      \left\Vert u_{t}(t) \right\Vert^{2}_{s-1-{\rho\over 2}}
                &=&  \left\Vert u_{t}(t) \right\Vert^{2}_{-1}= \int {1\over {1+\xi^{2}}}|\widehat{u}_{t}|^{2}d\xi \nonumber \\
                & \le &  \int {1\over {\xi^{2}}}|\widehat{u}_{t}|^{2}d\xi = \int | \widehat{\Lambda^{-1}u_{t}}|^{2}d\xi
                        =  \left\Vert \Lambda^{-1}u_{t} \right\Vert^{2} \le E(0)
                        \label{est-aa}
    \end{eqnarray}
    and
    \begin{eqnarray}
      \left\Vert u(t) \right\Vert^{2}_{s}
                &=&  \left\Vert u(t) \right\Vert^{2}_{{\rho\over 2}}= \int (1+\xi^{2})^{{\rho\over 2}}|\widehat{u}(\xi)|^{2}d\xi  \nonumber \\
                & \le &  {1\over c_{1}^{2}} \int k^{2}(\xi)|\widehat{u}(\xi)|^{2}d\xi
                        = {1\over c_{1}^{2}}\left\Vert Ku \right\Vert^{2}  \le {1\over c_{1}^{2}} E(0),
                        \label{est-bb}
    \end{eqnarray}
    where we have used (\ref{con-K}). Combining the two inequalities,  (\ref{est-aa}) and (\ref{est-bb}), we get
    \begin{displaymath}
    \limsup_{t \rightarrow T^{-}}  ~[ \left\Vert u(t) \right\Vert_{{\rho\over 2}}+\left\Vert u_{t}(t)\right\Vert_{-1}]
            \le \left(1+{1\over c_{1}}\right) \left(E(0)\right)^{1/2} < \infty .
    \end{displaymath}
    This is true for any $T>0$, therefore, $T_{\max}=\infty$ and we have the unique global solution
    $u(x,t)\in C\left( [0,\infty),H^{{\rho\over 2}}\right)\cap C^{1}\left( [0,\infty),H^{-1}\right)$.
\end{proof}
We now extend the above theorem to the general case.
\begin{theorem}\label{theo4.4}
     Assume that $\rho \ge 2$, $s>1/2$, $g\in C^{[s]+1}(\mathbb{R})$, $\varphi \in H^{s}$,
    $\psi\in H^{s-1-{\rho\over 2}}$,  $\Lambda^{-1}\psi \in L^{2}$, $K\varphi \in L^{2}$ and $G(\varphi) \in L^{1}$ and $G(u) \ge 0$ for all $u\in \mathbb{R}$. Then
    the Cauchy problem given by (\ref{non-boussinesq}) and (\ref{ini}) has a unique global solution
    $u\in C\left( [0,\infty),H^{s}\right)\cap C^{1}\left( [0,\infty),H^{s-1-{\rho\over 2}}\right)$.
\end{theorem}
\begin{proof}
     Assume that for some $T$, there exists a solution of the Cauchy problem on $[0, T)$. Then,  the estimate
     (\ref{lin-est}) yields
     \begin{equation}
     \left\Vert u(t) \right\Vert_{s}+\left\Vert u_{t}(t)\right\Vert_{s-1-{\rho\over 2}}
         \leq  A_{3}+B_{3} \int_{0}^{t}\left\Vert g(u(\tau)) \right\Vert_{s+1-{\rho\over 2}}d\tau   \label{}
    \end{equation}
    for all $t \in [0, T)$, where
    \begin{displaymath}
     A_{3}=(A_{1}+A_{2}T) \left(
            \left\Vert \varphi \right\Vert_{s}+\left\Vert \psi\right\Vert_{s-1-{\rho\over 2}}\right), ~~~~~~~B_{3}=A_{1}+A_{2}T.
    \end{displaymath}
    By the Sobolev Embedding Theorem and Theorem \ref{theo4.3}, we have
    \begin{equation}
        \left\Vert u(t) \right\Vert_{L^{\infty}} \le d \left\Vert u(t) \right\Vert_{{\rho\over 2}} \le {d\over c_{1}} \left(E(0)\right)^{1/2}.
     \end{equation}
    Using Lemma \ref{lem3.1} in this equation and noting that $\left\Vert u(\tau) \right\Vert_{s+1-{\rho\over 2}}\le  \left\Vert u(\tau) \right\Vert_{s}$
    we get
    \begin{eqnarray*}
    &&  \!\!\!\!\!\!\!\!\!\!\!
    \left\Vert u(t) \right\Vert_{s}+\left\Vert u_{t}(t)\right\Vert_{s-1-{\rho\over 2}}
                 \leq  A_{3}+B_{3}  C_{1}\left({d\over c_{1}} \left(E(0)\right)^{1/2}\right) \int_{0}^{t} \left\Vert u(\tau) \right\Vert_{s+1-{\rho\over 2}}d\tau,  \\
      &&~~~~~~~~~~~~~~~ ~~~~~ ~~~~~
      \le   A_{3}+B_{3}  C_{1}\left({d\over c_{1}} \left(E(0)\right)^{1/2}\right) \int_{0}^{t}
                        \left( \left\Vert u(\tau) \right\Vert_{s}+\left\Vert u_{t}(\tau) \right\Vert_{s-1-{\rho\over 2}} \right) d\tau.
    \end{eqnarray*}
    Applying Gronwall's Lemma to this inequality we find that
    $\left\Vert u(t) \right\Vert_{s}+\left\Vert u_{t}(t)\right\Vert_{s-1-{\rho\over 2}}$ stays bounded in $[0, t)$. This implies
    also that
    \begin{displaymath}
    \limsup_{t \rightarrow T^{-}} ~[ \left\Vert u(t) \right\Vert_{s}+\left\Vert u_{t}(t)\right\Vert_{s-1-{\rho\over 2}} ]
         < \infty.
    \end{displaymath}
    Thus, we conclude that $ T_{\max}=\infty$, i.e. there is a global solution.
\end{proof}

\setcounter{equation}{0}
\section{Blow-up in Finite Time}
\noindent
In this section we investigate finite time blow-up of solutions of the Cauchy problem given by (\ref{non-boussinesq}) and
(\ref{ini}). Our investigation relies on the following lemma, based on the idea of Levine \cite{levine}.
\begin{lemma}\label{lem5.1}
    \cite{kalantarov,levine}
    Suppose that $H(t)$, $t\ge 0$ is a positive, twice differentiable function satisfying
    $H^{\prime\prime}H-(1+\nu) \left(H^{\prime }\right)^{2}\geq 0$ where $\nu >0$. If $H(0)>0$ and $H^{\prime}(0) >0$, then
    $H(t) \rightarrow \infty $ as $t\rightarrow t_{1}$ for some  $t_{1}\leq H(0) /\left(\nu H^{\prime }(0)\right) $.
\end{lemma}
\begin{theorem}\label{theo5.2}
   Assume that $K\varphi \in L^{2}$, $\Lambda^{-1}\psi \in L^{2}$, $G(\varphi) \in L^{1}$.  If there is some $\nu >0$ such that
    \begin{equation}
    pg( p) \leq 2( 1+2\nu) G(p) \mbox{ for all }p\in \mathbb{R},  \label{blow-con}
    \end{equation}
    and
    \begin{displaymath}
    E(0) =\left\Vert \Lambda^{-1}\psi \right\Vert^{2}+\left\Vert K\varphi \right\Vert^{2}+2\int_{\Bbb R} G(\varphi) dx<0,
    \end{displaymath}
    then the solution $u(x,t)$ of the Cauchy problem given by (\ref{non-boussinesq}) and (\ref{ini}) blows up in finite time.
\end{theorem}
\begin{proof}
    Let
    \begin{displaymath}
    H(t)=  \left\Vert \Lambda^{-1}u(t) \right\Vert^{2}+b_{0}(t+t_{0})^{2}
    \end{displaymath}
    for some positive constants $b_{0}$ and $t_{0}$ to be determined later. By the energy identity it is clear that $H(t)$ is defined for all $t$ where the solution exists. Differentiation of this function leads to
    \begin{eqnarray}
    && H^{\prime}(t)= 2 \langle \Lambda^{-1}u, \Lambda^{-1}u_{t} \rangle+2b_{0}(t+t_{0}),   \label{first-der}\\
    && H^{\prime\prime}(t) =2 \left\Vert \Lambda^{-1}u_{t}\right\Vert^{2}
                                +2 \langle \Lambda^{-1}u, \Lambda^{-1}u_{tt} \rangle +2b_{0}. \label{second-der}
    \end{eqnarray}
    An application of the Cauchy-Schwarz inequality and $2ab \le a^{2}+b^{2}$ yields
    and
    \begin{eqnarray}
       [H^{\prime}(t)]^{2} & = & 4 \left[ \langle \Lambda^{-1}u, \Lambda^{-1}u_{t} \rangle+b_{0}(t+t_{0})\right]^{2} \nonumber \\
                & \le & 4 \left[ \left\Vert \Lambda^{-1}u\right\Vert \left\Vert \Lambda^{-1}u_{t}\right\Vert
                        +b_{0}(t+t_{0})\right]^{2} \nonumber \\
                & = & 4 \left[ \left\Vert \Lambda^{-1}u\right\Vert^{2} \left\Vert \Lambda^{-1}u_{t}\right\Vert^{2}
                        +b_{0}^{2}(t+t_{0})^{2}+2b_{0}(t+t_{0})
                        \left\Vert \Lambda^{-1}u\right\Vert \left\Vert \Lambda^{-1}u_{t}\right\Vert \right] \nonumber   \\
                & \le & 4 \left[ \left\Vert \Lambda^{-1}u\right\Vert^{2} \left\Vert \Lambda^{-1}u_{t}\right\Vert^{2}
                        +b_{0}^{2}(t+t_{0})^{2}+b_{0}\left\Vert \Lambda^{-1}u\right\Vert^{2}
                        +b_{0}\left\Vert \Lambda^{-1}u_{t}\right\Vert^{2}(t+t_{0})^{2}\right] \nonumber \\
                & = & 4 H(t) \left[ \left\Vert \Lambda^{-1}u_{t}\right\Vert^{2}+b_{0} \right]. \label{first-inequality}
    \end{eqnarray}
    Moreover, using (\ref{equ-motion}) in (\ref{second-der}) and then recalling the definition of $H(t)$  we get
    \begin{eqnarray*}
       H^{\prime\prime}(t) & = &  2 \left\Vert \Lambda^{-1}u_{t}\right\Vert^{2}
                                +2 \langle u, \Lambda^{-2}u_{tt} \rangle +2b_{0} \nonumber \\
                & = & 2 \left\Vert \Lambda^{-1}u_{t}\right\Vert^{2}
                                +2 \langle u, -K^{2}u-g(u) \rangle +2b_{0} \nonumber \\
                & = & 2 \left\Vert \Lambda^{-1}u_{t}\right\Vert^{2}
                                +2 \langle Ku, -Ku\rangle +2 \langle u,-g(u) \rangle +2b_{0} \nonumber \\
                & = & 2 \left\Vert \Lambda^{-1}u_{t}\right\Vert^{2}
                                -2 \left\Vert Ku\right\Vert^{2} -2 \int_{\mathbb{R}}u g(u)dx+2b_{0}.
    \end{eqnarray*}
    Combining this inequality with (\ref{blow-con}) gives
    \begin{equation}
        H^{\prime\prime}(t) \ge 2 \left\Vert \Lambda^{-1}u_{t}\right\Vert^{2}-2 \left\Vert Ku\right\Vert^{2}
            -4(1+2\nu)\int_{\mathbb{R}}G(u)dx+2b_{0} \label{second-inequality}
    \end{equation}
    Then, using (\ref{first-inequality}), (\ref{second-inequality}), (\ref{energy}) and recalling that $\nu>0$ we obtain
    \begin{eqnarray*}
     && \!\!\!\!\!\!\!\!\!\!\! H(t)H^{\prime\prime}(t)-(1+\nu)[H^{\prime}(t)]^{2} \nonumber \\
     && ~~~~~~~~~\ge  H(t)\left[H^{\prime\prime}(t)-4(1+\nu) \left( \left\Vert \Lambda^{-1}u_{t}\right\Vert^{2}+b_{0}\right) \right] \\
     && ~~~~~~~~~\ge  -2H(t)\left[(1+2\nu)\left( \left\Vert \Lambda^{-1}u_{t}\right\Vert^{2}
                        +2\int_{\mathbb{R}}G(u)dx+b_{0}  \right) + \left\Vert Ku\right\Vert^{2} \right] \\
     && ~~~~~~~~~\ge  -2(1+2\nu)H(t)\left( \left\Vert \Lambda^{-1}u_{t}\right\Vert^{2}+\left\Vert Ku\right\Vert^{2}
                        +2\int_{\mathbb{R}}G(u)dx+b_{0}  \right) \\
     && ~~~~~~~~~ = -2(1+2\nu)H(t) (E(0)+b_{0}).
    \end{eqnarray*}
    If we choose $b_{0}$ so that $0 \le b_{0} \le -E(0)$, then we obtain $H(t)H^{\prime\prime}(t)-(1+\nu)[H^{\prime}(t)]^{2}\ge 0$ and
    $H(0)>0$. Also, we choose $t_{0}$ sufficiently large so that $2\langle \Lambda^{-1}\varphi, \Lambda^{-1}\psi \rangle+2b_{0}t_{0}>0$,
    i.e.   $H^{\prime}(0)>0$. Lemma \ref{lem5.1} now allows us to conclude that $H(t)$ and thus the energy  blow up in finite time.
\end{proof}

\setcounter{equation}{0}
\section{A General Class of Double Dispersive Wave Equations}
\noindent
In this section, we return to our original motivation mentioned in the Introduction and consider the Cauchy problem
(\ref{non-equ})-(\ref{ini}). Below we mainly extend the results of the previous sections to the Cauchy problem
(\ref{non-equ})-(\ref{ini}). We will very briefly sketch the main ideas in the proofs leading to local-existence,
global existence and blow-up theorems since the analysis is similar in spirit to that of the previous sections.
The main point is to take into account the contribution of the additional term involving  $B$
in (\ref{non-equ}). Once again,  we remind the reader that $L$ and $B$ appearing in (\ref{non-equ}) are pseudodifferential
operators of order $\rho$ and $-r$, respectively.

We start with the local existence result for (\ref{non-equ})-(\ref{ini}).
\begin{theorem}\label{theo6.1}
     Assume that ${\rho\over 2}+ r \ge 1$, $s> {1\over 2}$, $\varphi \in H^{s}$, $\psi\in H^{s-1-{\rho\over 2}}$ and $g\in C^{[s]+1}(\mathbb{R})$.
     Then there is some $T>0$ such that the Cauchy problem (\ref{non-equ})-(\ref{ini}) has a unique  solution
    $u\in C\left( [0,T],H^{s}\right)\cap C^{1}\left( [0,T],H^{s-1-{\rho\over 2}}\right)$.
\end{theorem}
\begin{proof}
    As in Lemma \ref{lem3.3} and Theorem \ref{theo3.4}, the proof relies on the fixed point argument. Let $u$ be  the unique
    solution of the linearized problem
    \begin{eqnarray}
    && u_{tt}-Lu_{xx} =B\left(g(w)\right)_{xx},\quad x\in \mathbb{R},\quad t>0, \label{lin-equ-b}\\
    && u(x,0) =\varphi (x),\quad u_{t}(x,0)=\psi (x). \label{lin-ini-b}
    \end{eqnarray}
    Once again, the corresponding map will be denoted by $\cal S$, i.e. $u = {\cal S} (w)$. It follows from the estimate
    (\ref{con-B}) on the symbol $b$ of the operator $B$ that
    \begin{displaymath}
    \left\Vert Bv \right\Vert_{s} \le c_{3}^{2}\left\Vert v \right\Vert_{s-r}.
    \end{displaymath}
    Thus, using this inequality and Lemma \ref{lem3.1}  gives
    \begin{displaymath}
    \left\Vert Bg(w) \right\Vert_{s+1-{\rho\over 2}} \le c_{3}^{2}\left\Vert g(w) \right\Vert_{s+1-{\rho\over 2}-r}
                    \le c_{3}^{2} \left\Vert g(w) \right\Vert_{s} \le c_{3}^{2} C_{1}(M)\left\Vert w \right\Vert_{s}.
    \end{displaymath}
    Then the basic estimate for the above problem takes the following form
    \begin{displaymath}
     \left\Vert u(t) \right\Vert_{s}+\left\Vert u_{t}(t)\right\Vert_{s-1-{\rho\over 2}}
         \leq  (A_{1}+A_{2}T) \left(
            \left\Vert \varphi \right\Vert_{s}+\left\Vert \psi\right\Vert_{s-1-{\rho\over 2}}
            + c_{3}^{2}C_{1}(M)\int_{0}^{t}\left\Vert w \right\Vert_{s}d\tau \right).
    \end{displaymath}
    With this estimate, the rest of the proof follows exactly the same lines as those of Lemma \ref{lem3.3} and Theorem \ref{theo3.4}.
\end{proof}
We note that the key distinction between Theorem \ref{theo3.4} and Theorem \ref{theo6.1} is the condition
${\rho\over 2} + r \ge 1$ which is needed to overcome the two derivatives in the nonlinear term $g(u)_{xx}$.
We observe that, when the operator $B$ is simply the identity operator, this condition reduces to the one given in
Theorem \ref{theo3.4} where $b = 0$, thus $r=0$; so $\rho \ge 2$.
\begin{theorem}\label{theo6.2}
    Suppose that $u(x,t)$ is a solution of the Cauchy problem (\ref{non-equ})-(\ref{ini}) on some interval $[0, T)$.  Let $K=L^{1/2}$, $G(u) =\int_{0}^{u}g(p)dp$, $\Lambda^{-\alpha}w={\cal F}^{-1}[ |\xi|^{-\alpha}{\cal F}w]$ and $B^{-1/2}w={\cal F}^{-1}[ \left(b(\xi)\right)^{-1/2}{\cal F}w]$ where ${\cal F}$ and ${\cal F}^{-1}$ denote Fourier transform and inverse Fourier transform   in the $x$-variable, respectively. If $B^{-1/2}\Lambda^{-1}\psi \in L^{2}$, $B^{-1/2}K\varphi \in L^{2}$ and $G(\varphi) \in L^{1}$, then, for any $t\in [0, T)$, the energy identity
     \begin{equation}
     E(t)=\left\Vert B^{-1/2}\Lambda^{-1}u_{t} \right\Vert^{2}+\left\Vert B^{-1/2}Ku \right\Vert^{2}+2\int_{\mathbb{R}}G(u)dx=E(0)
        \label{dd-energy}
    \end{equation}
    is satisfied.
\end{theorem}
\begin{proof}
    The most important step in the proof is to rewrite (\ref{non-equ}) as
    \begin{displaymath}
    B^{-1}\Lambda^{-2}u_{tt}+B^{-1}K^{2}u+g(u)=0.
    \end{displaymath}
    Note that, when $B^{-1/2}\Lambda^{-1}$  and $B^{-1/2}K$ are replaced by $\Lambda^{-1}$  and $K$, respectively, this equation reduces (\ref{equ-motion}).  Noting this fact and recalling that  $B^{-1/2}$, $\Lambda^{-1}$  and $K$ are self-adjoint, we conclude that, not surprisingly, the proof follows
    that of Theorem \ref{theo4.2}.
\end{proof}

Once again, we give the global existence and uniqueness theorem  for two different regimes, i.e. for
$s={r\over 2}+{\rho\over 2}$ and a general $s$.
\begin{theorem}\label{theo6.3}
     Assume that $r+{\rho\over 2} \ge 1$, $s={r\over 2}+{\rho\over 2}$,
    $g\in C^{[{r\over 2}+{\rho\over 2}]+1}(\mathbb{R})$,
    $\varphi \in H^{{r\over 2}+{\rho\over 2}}$,
    $\psi\in H^{{r\over 2}-1}$, $\Lambda^{-1}\psi \in L^{2}$, $G(\varphi) \in L^{1}$ and $G(u) \ge 0$ for all $u\in \mathbb{R}$.
    Then the Cauchy problem (\ref{non-equ})-(\ref{ini}) has a unique global solution
    $u\in C\left( [0,\infty),H^{{r\over 2}+{\rho\over 2}}\right)\cap C^{1}\left( [0,\infty),H^{{r\over 2}-1}\right)$.
\end{theorem}
\begin{proof}
    The proof is entirely similar to that of Theorem \ref{theo4.3}. Assume that the solution exists for the times $[0, T)$. By using the condition $G(u)\ge 0$ in (\ref{dd-energy}), one obtains
    \begin{equation}
        \left\Vert B^{-1/2}\Lambda^{-1}u_{t} \right\Vert^{2}+\left\Vert B^{-1/2} Ku \right\Vert^{2} \le E(0). \label{dd-energy-inequality}
     \end{equation}
     Furthermore, note that
     \begin{eqnarray}
      \left\Vert u_{t}(t) \right\Vert^{2}_{{r\over 2}-1}
                &=&   \int (1+\xi^{2})^{{r\over 2}-1}|\widehat{u}_{t}|^{2}d\xi  \le \int {{(1+\xi^{2})^{{r\over 2}}}\over {\xi^{2}}}|\widehat{u}_{t}|^{2}d\xi\nonumber \\
                & \le & c_{3}^{2} \int {b^{-1}(\xi)\over {\xi^{2}}}|\widehat{u}_{t}|^{2}d\xi
                        = c_{3}^{2}  \left\Vert B^{-1/2}\Lambda^{-1}u_{t} \right\Vert^{2} \le c_{3}^{2}  E(0)
                        \label{dd-est-aa}
    \end{eqnarray}
    where we have used (\ref{con-B}) and (\ref{dd-energy-inequality}). Also note that combining (\ref{con-L}) and (\ref{con-B}) leads to
     \begin{equation}
        (1+\xi^{2})^{{r\over 2}+{\rho\over 2}} \le c_{1}^{-2}c_{3}^{2}b^{-1}(\xi)k^{2}(\xi).
     \end{equation}
     Using this inequality and (\ref{dd-energy-inequality}) we obtain
    \begin{eqnarray}
      \left\Vert u(t) \right\Vert^{2}_{{r\over 2}+{\rho\over 2}}
                &=&  \int (1+\xi^{2})^{{r\over 2}+{\rho\over 2}}|\widehat{u}(\xi)|^{2}d\xi \nonumber \\
                & \le &  {{c_{3}^{2}}\over c_{1}^{2}}\int b^{-1}(\xi)k^{2}(\xi)|\widehat{u}(\xi)|^{2}d\xi
                 =  {c_{3}^{2} \over c_{1}^{2}} \left\Vert B^{-1/2}Ku \right\Vert^{2} \le {c_{3}^{2} \over c_{1}^{2}}  E(0).
                        \label{dd-est-bb}
    \end{eqnarray}
    We then combine the two estimates,  (\ref{dd-est-aa}) and (\ref{dd-est-bb}), to get
    \begin{displaymath}
    \limsup_{t \rightarrow T^{-}}~ \left[ \left\Vert u(t) \right\Vert_{{r\over 2}+{\rho\over 2}}
        +\left\Vert u_{t}(t)\right\Vert_{{r\over 2}-1}\right]
            \le \left(c_{3}+{c_{3}\over c_{1}}\right) \left(E(0)\right)^{1/2} < \infty .
    \end{displaymath}
    With an argument similar to that in the proof of Theorem \ref{theo4.3} we conclude that $T_{\max}=\infty$ and that we have the global solution
    $u(x,t)\in C\left( [0,\infty),H^{{r\over 2}+{\rho\over 2}}\right)\cap C^{1}\left( [0,\infty),H^{{r\over 2}-1}\right)$.
\end{proof}
\begin{theorem}\label{theo6.4}
     Assume that $r+{\rho\over 2} \ge 1$, ${r\over 2}+{\rho\over 2}>1/2$, $s>1/2$, $g\in C^{[s]+1}(\mathbb{R})$, $\varphi \in H^{s}$,
    $\psi\in H^{s-1-{\rho\over 2}}$, $G(\varphi) \in L^{1}$ and $G(u) \ge 0$ for all $u\in \mathbb{R}$. Then
    the Cauchy problem (\ref{non-equ})-(\ref{ini}) has a unique global solution
    $u\in C\left( [0,\infty),H^{s}\right)\cap C^{1}\left( [0,\infty),H^{s-1-{\rho\over 2}}\right)$.
\end{theorem}
\begin{proof}
    Since  ${r\over 2}+{\rho\over 2}>1/2$, by the Sobolev Embedding Theorem we have $H^{{r\over 2}+{\rho\over 2}}\subset L^{\infty}$.
    Using this fact and (\ref{dd-est-bb}) yields
    \begin{equation}
        \left\Vert u(t) \right\Vert_{L^{\infty}} \le d \left\Vert u(t) \right\Vert_{{r\over 2}+{\rho\over 2}}
        \le d {c_{3}\over c_{1}}\left(E(0)\right)^{1/2}. \label{rest}
     \end{equation}
     Following similar steps as in the proof of Theorem \ref{theo4.4} and using (\ref{rest}) we obtain
    \begin{displaymath}
       \left\Vert u(t) \right\Vert_{s}+\left\Vert u_{t}(t)\right\Vert_{s-1-{\rho\over 2}}
                 \le   A_{3}+B_{3}  C_{1}\left(d {c_{3}\over c_{1}}\left(E(0)\right)^{1/2}\right) \int_{0}^{t}
                        \left( \left\Vert u(\tau) \right\Vert_{s}+\left\Vert u_{t}(\tau) \right\Vert_{s-1-{\rho\over 2}} \right) d\tau
    \end{displaymath}
    where the only difference with respect to the corresponding estimate of Theorem \ref{theo4.4} is the constant  $c_{3}$. Using
    Gronwall's Lemma and repeating the argument in Theorem \ref{theo4.4} we conclude that $ T_{\max}=\infty$, i.e. there is a global solution.
\end{proof}

We now investigate finite time blow-up of solutions of the Cauchy problem (\ref{non-equ})-(\ref{ini}). Once again, our investigation relies on Lemma \ref{lem5.1}.
\begin{theorem}\label{theo6.5}
    Assume that $B^{-1/2}K\varphi \in L^{2}$, $B^{-1/2}\Lambda^{-1}\psi \in L^{2}$, $G(\varphi) \in L^{1}$.  If there is some $\nu >0$ such that
    \begin{equation}
    pg( p) \leq 2( 1+2\nu) G(p) \mbox{ for all }p\in \mathbb{R},  \label{dd-blow-con}
    \end{equation}
    and
    \begin{displaymath}
    E(0) =\left\Vert B^{-1/2}\Lambda^{-1}\psi \right\Vert^{2}+\left\Vert B^{-1/2}K\varphi \right\Vert^{2}
    +2\int_{\Bbb R} G(\varphi) dx<0,
    \end{displaymath}
    then the solution $u(x,t)$ of the Cauchy problem (\ref{non-equ})-(\ref{ini}) blows up in finite time.
\end{theorem}
\begin{proof}
    The proof is very similar to that of Theorem \ref{theo5.2}. The main differences are that we now use the function
    \begin{displaymath}
    H(t)=  \left\Vert B^{-1/2}\Lambda^{-1}u(t) \right\Vert^{2}+b_{0}(t+t_{0})^{2}
    \end{displaymath}
    and that the energy $E(t)$ has a different form for the double dispersive equations. Detailed calculations are not
    presented here, but we can easily derive them  replacing $\Lambda^{-1}$  and $K$ in the proof of Theorem \ref{theo5.2} by
    $B^{-1/2}\Lambda^{-1}$  and $B^{-1/2}K$, respectively. Similar computations as in the proof of Theorem \ref{theo5.2} establish
    \begin{displaymath}
       [H^{\prime}(t)]^{2}
                \le  4 H(t) \left[ \left\Vert B^{-1/2}\Lambda^{-1}u_{t}\right\Vert^{2}+b_{0} \right],
    \end{displaymath}
    and
    \begin{displaymath}
       H^{\prime\prime}(t)  =  2 \left\Vert B^{-1/2}\Lambda^{-1}u_{t}\right\Vert^{2}
                                -2 \left\Vert B^{-1/2}Ku\right\Vert^{2} -2 \int_{\mathbb{R}}u g(u)dx+2b_{0}.
    \end{displaymath}
    Using (\ref{dd-blow-con}) in this equation we deduce that
    \begin{displaymath}
        H^{\prime\prime}(t) \ge 2 \left\Vert B^{-1/2}\Lambda^{-1}u_{t}\right\Vert^{2}-2 \left\Vert B^{-1/2}Ku\right\Vert^{2}
            -4(1+2\nu)\int_{\mathbb{R}}G(u)dx+2b_{0}.
    \end{displaymath}
    and that
    \begin{displaymath}
       H(t)H^{\prime\prime}(t)-(1+\nu)[H^{\prime}(t)]^{2} \ge  -2(1+2\nu)H(t) (E(0)+b_{0})
    \end{displaymath}
    With an argument similar to that of Theorem \ref{theo5.2} we conclude that $H(t)$ and the energy blow up in finite time.
\end{proof}

\vspace*{10pt}

\noindent
{\bf Acknowledgement}: This work has been supported by the Scientific and Technological Research Council of Turkey
(TUBITAK) under the project TBAG-110R002.


 \end{document}